\documentclass[11pt]{article}

\usepackage[left=30mm,right=30mm, top=30mm, bottom=30mm]{geometry}
\usepackage{graphicx,color}
\usepackage[utf8]{inputenc}
\usepackage{amsmath}
\usepackage{amssymb}
\usepackage{amsthm}
\usepackage{comment}
\usepackage{enumerate}
\usepackage{mathtools}
\usepackage{cases}
\usepackage{multicol}
\usepackage{tabularx}
\usepackage{authblk}
\usepackage{todonotes}
\usepackage[colorlinks]{hyperref}
\hypersetup{
    colorlinks=true,
    linkcolor=red,
    citecolor=red
}

\theoremstyle{plain}
\newtheorem{thm}{Theorem}[section]
\newtheorem{prop}[thm]{Proposition}

\newtheorem{cor}[thm]{Corollary}

\theoremstyle{definition}
\newtheorem{defi}[thm]{Definition}
\newtheorem{ex}[thm]{Example}

\theoremstyle{remark}
\newtheorem{remark}[thm]{Remark}

\newcommand{\C}{\mathbb C}
\newcommand{\R}{\mathbb R}
\newcommand{\N}{\mathbb N}
\newcommand{\X}{\mathcal{X}}
\newcommand{\Y}{\mathcal{Y}}
\newcommand{\Z}{\mathcal{Z}}

\newcommand{\dom}{\mathcal{D}}

\newcommand{\W}{\mathcal{W}}
\newcommand{\bem}{\begin{bmatrix}}
\newcommand{\enm}{\end{bmatrix}}
\newcommand{\e}{\mathrm{e}}
\renewcommand{\i}{\mathrm{i}}
\newcommand{\dx}[1][x]{\,\mathrm{d}#1}

\newcommand{\Real}{{\rm Re}}

\newcommand{\ran}{{\rm ran}\,}
\renewcommand{\ker}{{\rm ker}\,}

\newcommand{\inv}{^{-1}}
\newcommand{\ainv}{^{-\ast}}
\newcommand{\set}[1]{\mathopen{}\left\{#1\right\}\mathclose{} }
\renewcommand{\mid}{\,\middle|\,}

\title{On the Weierstraß form of infinite dimensional differential algebraic equations}

\author[1]{Mehmet Erbay}
\author[2]{Birgit Jacob}
\author[3]{Kirsten Morris}

\affil[1]{IMACM, Univ. of Wuppertal, Gaußstraße 20, 42119 Wuppertal, Germany, {\tt erbay@uni-wuppertal.de}}
\affil[2]{IMACM, Univ. of Wuppertal, Gaußstraße 20, 42119 Wuppertal, Germany, {\tt bjacob@uni-wuppertal.de}}
\affil[3]{Dept. of Applied Mathematics, Univ. of Waterloo, 200 University Avenue West, Waterloo, ON N2L 3G1, Canada, {\tt kmorris@uwaterloo.ca}}

\date{\today}

\begin{document}
\maketitle

\begin{abstract}
    The solvability for infinite dimensional differential algebraic equations possessing a resolvent index and a Weierstraß form is studied. In particular, the concept of integrated semigroups is used to determine a subset on which solutions exist and are unique. This information is later used for a important class of systems, namely, port-Hamiltonian differential algebraic equations.
\end{abstract}

\section{Introduction}
Linear differential algebraic equations (DAEs), sometimes also called \textit{descriptor systems} or \textit{implicit differential equations}, arise in various fields such as physics, engineering and economics. They can be written as
\begin{equation}\label{eqn:dae-1}
    \frac{d}{dt}Ex(t)=Ax(t), \quad t\geq 0.
\end{equation}
Here, for complex Hilbert spaces $\X$ and $\Z$, $E\colon\X\to\Z$ is a bounded operator (denoted by $E\in L(\X,\Z)$) and $A\colon\dom(A)\subseteq \X \to\Z$ is a closed and densely defined operator. We will often write $(E,A)$ to refer to \eqref{eqn:dae-1}. 

In comparison to ordinary differential equations, DAEs include both algebraic and differential constraints. Because of this, in general $E$ has a non-trivial kernel. In finite dimensions it is always possible to decouple the system into an algebraic and a differential part by transforming \eqref{eqn:dae-1} into a so called \textit{Weierstraß form} 
\begin{equation}\label{eqn:dae-2}
    \frac{d}{dt}\bem I & 0 \\ 0 & N \enm \bem x_1(t)\\ x_2(t)\enm = \bem J & 0 \\ 0 & I\enm \bem x_1(t) \\ x_2(t)\enm, \quad t\geq 0,
\end{equation}
where $J$ is in Jordan form and $N$ is nilpotent, as long as $(\lambda E-A)\inv \in L(\Z,\X)$ for some $\lambda \in \C$,
\cite[Ch.~2.1]{kunkel_differential-algebraic_2006}. In this context, the nilpotency degree of $N$ is called the \textit{nilpotency index} of $(E,A)$ \cite{erbay_index_2024, kunkel_differential-algebraic_2006}. For finite-dimensional spaces, this definition of index coincides with other common definitions of the index. In this context, investigation of the different index definitions can be found in \cite{gernandt_pseudo-resolvent_2023, trostorff_semigroups_2020, trostorff_higher_2018} for the \textit{resolvent index}, \cite{jacob_solvability_2022, sviridyuk_linear_2003} for the radiality index, \cite{erbay_index_2024} for a general comparison of different known indices in infinite-dimensions and \cite{kunkel_differential-algebraic_2006, mehrmann_index_2015} for the finite-dimensional case. 

However, in infinite-dimensions the existence of a Weierstraß form is not guaranteed, nor is there a general procedure for calculating it. In view of this difficulty, we will specify a condition under which such a form always exists, namely the existence of the \textit{radiality index}.
One of the main uses of the Weierstraß form is to analyse the well-posedness of the system, since the DAE is divided into an algebraic part and an ODE part, see \eqref{eqn:dae-2}. For the ODE part it is possible to use known solution methods. 
The solvability of infinite-dimensional DAEs has been intensively studied; see for example, \cite{gernandt_pseudo-resolvent_2023, jacob_solvability_2022, mehrmann_abstract_2023, melnikova_properties_1996, reis_controllability_2008, reis_frequency_2005, thaller_factorization_1996, thaller_semigroup_2001, trostorff_semigroups_2020, trostorff_higher_2018}. 
In \cite{thaller_factorization_1996} the splitting of $\X$ into $\ker E$ and $\ran E^\ast$ and the restriction of the DAE to the factorized space $\X / \ker E$ is studied, in \cite{jacob_solvability_2022, sviridyuk_linear_2003} and \cite{thaller_semigroup_2001} the splitting of the space, given that the radiality index exists, and the growth of the pseudo-resolvents $(\lambda E-A)\inv E$ and $E(\lambda E-A)\inv$ are analysed, in \cite{gernandt_pseudo-resolvent_2023} a more general observation of the pseudo-resolvent growth is provided and a useful dissipativity condition for solving the DAE is presented, in \cite{trostorff_semigroups_2020} and \cite{trostorff_higher_2018} with the use of Wong-sequences and the resolvent index a super-set of the solution space is determined, in \cite{reis_frequency_2005} a sufficient condition in terms of Hille-Yosida type resolvent estimates is offered, in \cite{melnikova_properties_1996} with the help of integrated semigroups the well-posedness is investigated and in \cite{mehrmann_abstract_2023} the solvability of a class of \textit{port-Hamiltonian DAEs}, is discussed.

As in \cite{melnikova_properties_1996}, we will study the well-posedness of \eqref{eqn:dae-1} using integrated semigroups. There is a strong connection between integrated semigroups and the growth rate of the resolvent $(\lambda E-A)\inv$, for $\lambda \in \rho(E,A)\neq \emptyset$, where
\begin{equation*}
    \rho(E,A)\coloneqq \set{\lambda \in \C \mid (\lambda E-A)\inv \in L(\X,\Z)}.
\end{equation*}
A densely defined linear operator $A$ generates an integrated semigroup if there exists an $n\in \N$, such that
\begin{equation}\label{eqn:int-semi}
    \left\Vert (\lambda I -A)\inv \right\Vert \leq C \left\vert \lambda\right\vert^{n-1}, \quad \lambda \in \C_{\Real\geq \omega},
\end{equation}
for some $C, \omega >0$ \cite[Thm.~4.8]{neubrander_integrated_1988}. For $n=0$ this matches the condition of a sectorial operator operator $A$, which is used for the generation of a analytic semigroup.
Notice, that \eqref{eqn:int-semi} coincides with the \textit{complex resolvent index} for $E=I$. Based on this, we build on existing solution methods for the abstract Cauchy problem $\frac{d}{dt}x(t)=Ax(t)$, $x(0)=x_0$ on the subspace $\dom(A^n) = \ran ((\lambda I-A)\inv)^n$ \cite[§1.3]{krejn_linear_1971} and extend these studies to the differential algebraic equation
\eqref{eqn:dae-1} with $x_0 \in \ran( (\lambda E-A)\inv E)^n$. For further details on integrated semigroups we refer to \cite{arendt_vector-valued_2011, neubrander_integrated_1988}.

We subsequently focus on a special class of systems, namely, \textit{port-Hamiltonian differential algebraic equations} (pH-DAEs). These systems provide a framework for modeling and analyzing energy-dissipating physical systems, including electrical circuits, mechanical systems, fluid dynamic, thermal systems or robotics and control systems (just to name a few). There are different ways to define and approach pH-DAEs, such as using relations \cite{gernandt_linear_2021} or Dirac-structures \cite{van_der_schaft_generalized_2018}. Here, we use an operator formulation, as in \cite{mehrmann_abstract_2023}. To be more specific, we consider
\begin{equation}\label{eqn:dae-3}
    \frac{d}{dt}Ex(t)=AQx(t),\quad t\geq 0
\end{equation}
with $E,Q\in L(\X,\Z)$, $Q$ invertible, $A\colon\dom(A)\subseteq \Z\to\Z$ closed, densely defined and dissipative and $E^\ast Q \geq 0$ non-negative, i.e.~$\langle E^\ast Q x,x\rangle\geq 0$ for all $x\in \X$. The Hamiltonian of the pH-DAE \eqref{eqn:dae-3} is given by $\langle E^\ast Q x,x\rangle$.
Additionally to this, we will assume that $\ran E$ is closed. This guarantees the dissipation of the energy
\begin{equation}
    \frac{d}{dt}\langle E^\ast Q x(t), x(t)\rangle \leq 0,
\end{equation}
along all classical solutions of \eqref{eqn:dae-3}. 

The paper is organized as follows. In \hyperref[section-2]{Section 2} a set of initial conditions for which \eqref{eqn:dae-1} has a solution is characterized. In \hyperref[section-3]{Section 3} the radiality index is defined and it is shown that if $\ran E$ is closed, existence of this index implies existence of a Weierstraß form. Several new results relating the radiality index and the degree of nilpotency are also obtained.
Furthermore, provided that the complex resolvent index exists, and the Weierstraß form exists with densely defined $A_1$, then $A_1$ generates an integrated semigroup.

In \hyperref[section-4]{Section 4} port-Hamiltonian DAEs (pH-DAEs) are formally described and their dissipativity is proven. Solvability of port-Hamiltonian DAEs is studied in \hyperref[section-5]{Section 5}. With the aid of the results earlier in the paper, it is shown that pH-DAEs have complex resolvent index at most 3. This extends a result that is known in finite-dimensions to the infinite-dimensional setting. Existence of solutions to pH-DAEs is then obtained, along with generation of an integrated semigroup.

\section{Existence of solutions on a subspace}\label{section-2}
Throughout this article, let $\X$ and $\Z$ be Hilbert spaces, and, if not mentioned otherwise, $A\colon \dom(A)\subseteq \X\to\Z$ is a closed an densely defined operator and $E\in L(\X,\Z)$. In this section, we prove that if that the complex resolvent index exists, solutions exist on a subspace. We first recall the definition of the complex resolvent index, see \cite{erbay_index_2024, trostorff_semigroups_2020}.

\begin{defi}
    The \textit{(complex) resolvent index} of \eqref{eqn:dae-1} is the smallest number $p=p_{\rm res}^{(E,A)}\in\N$ $(p=p_{\rm c,res}^{(E,A)})$, such that there exists a $\omega \in\R$, $C>0$ with $(\omega,\infty)\subseteq \rho(E,A)$ $(\C_{\Real>\omega}\subseteq \rho(E,A))$ and 
    \begin{equation}\label{eqn:res-defi}
        \Vert (\lambda E-A)\inv \Vert \leq C\vert\lambda\vert^{p-1},\quad \lambda\in(\omega,\infty) \; (\lambda \in \C_{\Real>\omega}).
    \end{equation}
\end{defi}

In the following we determine a set of initial conditions for which the differential algebraic equation \eqref{eqn:dae-1} has a classical solution.

In this context, we say that $x\colon \R_{\geq 0}\to\X$ is a \textit{classical solution} of \eqref{eqn:dae-1} for an initial condition $x_0\in \X$, if $x(\cdot)$ is continuous on $\R_{\geq 0}$, $Ex(\cdot)$ is continuously differentiable on $\R_{>0}$, $x(t)\in \dom(A)$ for all $t\geq 0$, $x(0)=x_0$ and \eqref{eqn:dae-1} holds, and we say that $x\colon \R_{\geq 0}\to \X$ is a \textit{mild solution}, if $Ex(\cdot)$ is continuous, and for almost all $t\geq 0$, it holds $\int_0^t x(s)\dx[s] \in \dom(A)$ with
    \begin{equation*}
        Ex(t)- Ex_0 = A\int_0^t x(s)\dx[s].
    \end{equation*}

The following theorem is a generalization of the solution theory of the abstract Cauchy problem $\frac{d}{dt} x(t)=Ax(t)$, $x(0)=x_0$ and extends the proofs of \cite[p.~34]{krejn_linear_1971} and \cite[Ch.~4, Thm.~1.2]{pazy_semigroups_1983} to the DAE case with the help of pseudo-resolvents.

\begin{thm}\label{thm:7}
    Assume, that $(E,A)$ has a complex resolvent index. Then, for every $x_0 \in \ran ((\mu E-A)\inv E)^{p_{\mathrm{c,res}}^{(E,A)}+2}$ there exists a unique classical solution $x\colon \R_{\geq 0}\to\X$ of \eqref{eqn:dae-1}.
\end{thm}

\begin{proof}
    \textit{Existence}: 
    Let $p =p_{\mathrm{c,res}}^{(E,A)}+2$ and, for the sake of simplicity, define $R(\mu) \coloneqq ((\mu E-A)\inv E)$ for $\mu \in \rho(E,A)$. By assumption, the complex resolvent index exists. Thus, there are $\omega,C >0$, such that 
    \begin{equation}\label{eqn:res1}
        \left\Vert R(\lambda) \right\Vert \leq C \left\vert \lambda\right\vert^{p_{\mathrm{c,res}}^{(E,A)}-1} \left\Vert E \right\Vert, \quad \lambda \in \C_{\Real \geq  \omega}.
    \end{equation}
    Let $x_0 \in \ran R(\mu)^p$. Thus, there exists a $z_0\in \X$ and a $\mu \in \C_{\Real > \omega}$ such that $x_0= (-1)^{p-1} R(\mu)^{p}z_0$
    and therefore
    \begin{equation}
        \left\Vert \frac{R(\lambda)}{(\lambda-\mu)^p}\right\Vert \leq \tilde C \frac{\left\vert \lambda \right\vert^{p-3}}{\left\vert \lambda -\mu\right\vert^p},
    \end{equation}
    for a $\tilde C>0$ and for all $\lambda \neq \mu$. Thus, the integral of $\e^{\lambda t}\frac{R(\lambda)z_0}{(\lambda-\mu)^p}$ along the imaginary axis at $\Real \,\lambda = \omega$ exists and, especially, the function 
    \begin{equation}\label{eqn:solution}
        x(t)\coloneqq 
        \frac{(-1)^p}{2\pi i} \int_{\omega-i\infty}^{\omega+i\infty} \e^{\lambda t} \frac{R(\lambda)z_0}{(\lambda-\mu)^p}\dx[\lambda], \quad t\geq 0,
    \end{equation}
    is continuously differentiable. Since $E$ is bounded, the function $Ex(t)$ becomes differentiable as well and
    \begin{align}
        \frac{d}{dt}Ex(t) &= \frac{(-1)^p}{2\pi i} \int_{\omega-i\infty}^{\omega+i\infty} \lambda \e^{\lambda t} \frac{E R(\lambda)z_0}{(\lambda-\mu)^p}\dx[\lambda] \\
        &= \frac{(-1)^p}{2\pi i} \int_{\omega-i\infty}^{\omega+i\infty} \e^{\lambda t} \frac{AR(\lambda)z_0}{(\lambda-\mu)^p}\dx[\lambda] + \underbrace{\frac{(-1)^p}{2\pi i} \int_{\omega-i\infty}^{\omega+i\infty} \e^{\lambda t} \frac{Ez_0}{(\lambda-\mu)^p}\dx[\lambda]}_{=0} = A x(t).\label{eqn:soltion2}
    \end{align}
    Here, we used that $\frac{Ez_0}{(\lambda-\mu)^p}$ converges uniformly to $0$ for $\lambda \to \infty$ and Jordan's Lemma to show that the second integral in \eqref{eqn:soltion2} vanishes.
    Furthermore, 
    \begin{equation*}
        x(0) = \frac{(-1)^{p}}{2\pi i} \int_{\omega -i \infty}^{\omega + i \infty} \frac{R(\lambda)z_0}{(\lambda-\mu)^p}\dx[\lambda].
    \end{equation*}
    Using again Jordan's Lemma and, additionally, the residue theorem, we derive
    \begin{align*}
        x(0) = \frac{(-1)^{p-1}}{(p-1)!} \lim_{\lambda \to \mu} \frac{d^{p-1}}{d \lambda^{p-1}} R(\lambda) z_0 = (-1)^{p-1} R(\mu)^p z_0 = x_0.
    \end{align*}
    Here, we used $\frac{d}{d\lambda} R(\lambda)^n z_0 = - n R(\lambda)^{n+1} z_0$ for all $n \in \N$ and $\lambda \in \rho(E,A)$ (\cite[Lem.~2.4.1]{sviridyuk_linear_2003}). Thus, $x$ is a classical solution of \eqref{eqn:dae-1}.
    
    \textit{Uniqueness:} We assume that $x$ is a solution with $x(0) = 0$ and prove that $x(t) = 0$ for all $t\geq 0$. Since $x$ is a solution of \eqref{eqn:dae-1}, we deduce
    \begin{equation*}
        \frac{d}{dt} R(\lambda) x(t) = (\lambda E -A)\inv A x(t) = \lambda R(\lambda) x(t) - x(t),
    \end{equation*}
    for an $\lambda \in \rho(E,A)$. The solution of this equation can be written as
    \begin{equation}\label{eqn:res2}
        R(\lambda) x(t) = \e^{\lambda t } x(0) - \int_0^t \e^{\lambda (t-\tau)}x(\tau)\dx[\tau] = -\int_0^t \e^{\lambda (t-\tau)}x(\tau)\dx[\tau].
    \end{equation}
    Let $\sigma>0$. Inequality \eqref{eqn:res1} then implies
    \begin{equation}\label{eqn:res3}
        \lim_{\Real \lambda \to \infty} \e^{-\sigma \lambda} \left\Vert R(\lambda) \right\Vert = 0.
    \end{equation}
    By \eqref{eqn:res2}, \eqref{eqn:res3} and the fact that a solution $x$ is bounded on $[0,t]$ we conclude
    \begin{align*}
        \lim_{\Real \lambda \to \infty} \bigg\Vert \int_0^{t-\sigma} & \e^{\lambda (t-\sigma-\tau)} x(\tau)\dx[\tau] \bigg\Vert \\
        &= \lim_{\Real \lambda \to \infty} \e^{-\sigma \Real \lambda} \left\Vert \int_0^{t-\sigma} \e^{\lambda (t-\tau)}x(\tau)\dx[\tau]\right\Vert \\
        &\leq  \lim_{\Real \lambda \to \infty} \e^{-\sigma \Real\lambda} \bigg\Vert \underbrace{\int_0^t \e^{\lambda (t-\tau)}x(\tau)\dx[\tau]}_{=R(\lambda) x(t)} \bigg\Vert + \underbrace{\e^{-\sigma \Real\lambda} \left\Vert \int_{t-\sigma}^t \e^{\lambda (t-\tau)}x(\tau)\dx[\tau]\right\Vert}_{\to0} = 0.
    \end{align*}
    Thus,
    \begin{equation*}
        \lim_{\Real \lambda \to\infty} \int_0^{t-\sigma} \e^{\lambda(t-\sigma-\tau)} x(\tau)\dx[\tau] = 0
    \end{equation*}
    and by \cite[Ch.~4, Lem.~1.1]{pazy_semigroups_1983} $x(\tau) = 0$ for all $0\leq \tau \leq t-\sigma$. Since $\sigma$ and $t$ were arbitrary, we deduce $x(\tau) = 0$ for all $\tau\geq 0$.
\end{proof}

\begin{remark}
\begin{enumerate}[\;\, a)]
    \item Of course it is also possible to formulate the initial condition $x(0)=x_0$ of \eqref{eqn:dae-1}
    as
    \begin{equation*}
        \begin{cases}
            \frac{d}{dt}Ex(t) &= Ax(t),\quad t\geq 0,\\
            Ex(0) &= x_0.
        \end{cases}
    \end{equation*}
    Accordingly, one must assume $x_0\in E(\ran ((\mu E-A)\inv E)^{p_{\mathrm{c,res}}^{(E,A)}+2})$ in Theorem \ref{thm:7}. 
    
    \item The term $p=p_{\mathrm{c,res}}^{(E,A)}+2$ in the proof of Theorem \ref{thm:7} is chosen in such a way that $\frac{R(\lambda)}{(\lambda-\mu)^p}$ falls quickly enough, such that the integral in \eqref{eqn:solution} not only exists, the function $x(t)$ also becomes continuously differentiable. This means that the function $x$ in \eqref{eqn:solution} is only continuous, if $x_0 \in \ran ((\mu E-A)\inv E)^{p_{\mathrm{c,res}}^{(E,A)}+1}$. However, in this case $x$ is a mild solution of \eqref{eqn:dae-1}, as one can compute
    \begin{align*}
        Ex(t)- Ex(0) &= -\frac{1}{2\pi i}\int_{\omega -i \infty}^{\omega + i \infty} \underbrace{(\e^{\lambda t }-1)}_{=\int_0^t \lambda \e^{\lambda s}\dx[s]} \frac{E R(\lambda)z_0}{(\lambda-\mu)^p}\dx[\lambda]\\
        &= \int_0^t \bigg( -\frac{1}{2\pi i}\int_{\omega -i \infty}^{\omega + i \infty} \lambda \e^{\lambda s} \frac{ER(\lambda)z_0}{(\lambda-\mu)^p}\dx[\lambda] \bigg)\dx[s]\\
        &= \int_0^t \bigg( -\frac{1}{2\pi i} \int_{\omega-i\infty}^{\omega+i\infty} \e^{\lambda s} \frac{AR(\lambda)z_0}{(\lambda-\mu)^p}\dx[\lambda] - \underbrace{\frac{1}{2\pi i} \int_{\omega-i\infty}^{\omega+i\infty} \e^{\lambda s} \frac{Ez_0}{(\lambda-\mu)^p}\dx[\lambda]}_{=0} \bigg)  \dx[s] \\
        &= \int_0^t Ax(s)\dx[s] = A \int_0^t x(s)\dx[s].
    \end{align*}

    \item In \cite[Thm.~5.7]{trostorff_semigroups_2020} the set of initial values
    \begin{equation}\label{trostorff}
        U\coloneqq\set{x_0 \in \X \mid \exists x\colon \R_{\geq 0}\to \X \text{ mild solution}}
    \end{equation}
    was analyzed in more detail. To be more precise, it was shown that $E\inv A$ generates a $C_0$-semigroup on $\overline{\ran ((\mu E-A)\inv E)^{p_{\rm c,res}^{(E,A)}+2}}$ if and only if $U$ is closed and for each $x_0 \in U$ the mild solution is unique. The assumption that $U$ is closed is a bit restrictive and it also requires information about $U$.
    
    \item In \cite[Section 8]{gernandt_pseudo-resolvent_2023} the existence of solutions on the subspace $\overline{\ran R(\lambda)^k}$, for $R(\lambda)\coloneqq (\lambda E-A)\inv E$, has been studied under slightly stronger conditions compared to the ones in Theorem \ref{thm:7}. More specifically, under the assumption that there exists $k\in \N$ with $k\geq 1$, $\omega\in \R$ and $M>0$, such that $(\omega,\infty)\subseteq \rho(E,A)$ and 
    \begin{equation*}
        \left\Vert R(\lambda)x\right\Vert \leq \frac{M}{\lambda -\omega} \left\Vert x\right\Vert, \quad \forall \lambda >\omega, \; x\in \ran R(\omega)^{k-1}
    \end{equation*}
    the following operator $A_R$ was defined through its graph
    \begin{equation*}
        {\rm{graph}}\, A_R = \set{ (R(\lambda) x, x+\lambda R(\lambda)x)\mid x\in \overline{\ran R(\lambda)^k}}.
    \end{equation*}
    Then, under the assumption that $A_R$ generates a $C_0$-semigroup, it was shown that for every $x_0\in \overline{\ran R(\lambda)^k}$ there exists a unique mild solution of \eqref{eqn:dae-1} and for every $x_0 \in R(\lambda)\overline{\ran R(\lambda)^k}$ there exists a classical solution of \eqref{eqn:dae-1}. As it was remarked by the authors, the generation of the $C_0$-semigroup can be achieved if 
    \begin{equation*}
        \left\Vert \lambda R(\lambda) x\right\Vert \leq M \left\Vert x\right\Vert \quad \forall \lambda \in \C_{\geq \omega}, \; x\in \ran R(\omega)^k
    \end{equation*}
    holds, which is a slightly stronger assumption compared to the complex resolvent index. In the latter case, we only have
    \begin{equation*}
        \left\Vert \lambda R(\lambda) R(\mu)^k x \right\Vert \leq M \left\Vert x\right\Vert \quad \forall \lambda \in \C_{\geq \omega},\; x\in \X,
    \end{equation*}
    for $k=p_{\rm c,res}^{(E,A)}$.
\end{enumerate}
\end{remark}

\section{Sufficient condition for a Weierstraß form}\label{section-3}
In this section we start with the analysis of a sufficient condition for the existence of a Weierstraß form. To do that, we recall the definition of the radiality index \cite{erbay_index_2024} or, more generally, of the radiality \cite{sviridyuk_linear_2003}. 

We call two differential-algebraic equations $\frac{d}{dt}Ex=Ax$ and $\frac{d}{dt}\tilde E \tilde x=\tilde A\tilde x$ defined on Hilbert spaces $\X$,$\Z$ and $\tilde \X$, $\tilde \Z$ \textit{equivalent}, denoted by $(E,A)\sim(\tilde E,\tilde A)$, if there are two bounded isomorphisms $P\colon \X\to\tilde \X$, $Q\colon \Z\to\tilde \Z$, such that $E=Q\inv \tilde E P$ and $A=Q\inv \tilde A P$. Based on that, we define the notion of a Weierstraß form as follows. 
\begin{defi}
    The system \eqref{eqn:dae-1} has a \textit{Weierstraß form}, if there exists a Hilbert space $\Y=\Y^1\oplus\Y^2$, such that
    \begin{equation}\label{eqn:dae-wform}
        (E,A) \sim \left(\begin{bmatrix}
            I_{\Y^1} & 0\\ 0 & N
        \end{bmatrix}, \begin{bmatrix}
            A_1 & 0 \\ 0 & I_{\Y^2}
        \end{bmatrix}\right),
    \end{equation}
    where $N\colon \Y^2\to\Y^2$ is a nilpotent operator, $A_1\colon\dom(A_1)\subseteq \Y^1\to\Y^1$ is a linear operator and $I_{\Y^i}$ indicates the identity operator on the associated subspace $\Y^i$, $i=1,2$. Furthermore, the nilpotency degree of $N$ is known as the \textit{nilpotency index} $p_{\rm nilp}^{(E,A)}$ \cite[Section 2]{erbay_index_2024}.
\end{defi}

\begin{defi}
    \begin{enumerate}[\;\, a)]
        \item The system \eqref{eqn:dae-1} has a \textit{(complex) radiality index}, if there exists a $p=p_{\rm rad}^{(E,A)}\in \N$ $(p=p_{\rm c, rad}^{(E,A)}\in \N)$ and $\omega, C>0$, such that $(\omega, \infty)\subseteq \rho(E,A)$ ($\C_{\Real>\omega}\subseteq\rho(E,A)$) and 
        \begin{align}\label{rad-index}
            \begin{split}
                \left\Vert \left((\lambda_0E-A)\inv E\cdot\ldots\cdot(\lambda_p E-A)\inv E\right)^n\right\Vert 
                \leq \frac{C}{\prod_{k=0}^{p}\vert\lambda_k-\omega\vert^n},\\
                \left\Vert \left(E(\lambda_0E-A)\inv \cdot\ldots\cdot E(\lambda_p E-A)\inv\right)^n \right\Vert 
                \leq \frac{C}{\prod_{k=0}^{p}\vert\lambda_k-\omega\vert^n},
            \end{split}
        \end{align}
        holds for all $\lambda_0,\ldots,\lambda_{p}\in (\omega,\infty)$ $(\lambda_0,\ldots,\lambda_{p}\in \C_{\Real>\omega})$ and $n=1$.
        \item The system \eqref{eqn:dae-1} is \textit{$p$-radial}, if it has radiality index $p$ and \eqref{rad-index} holds for all $n\in \N$.
    \end{enumerate}
\end{defi}

\begin{thm}\label{thm:1}
    Let $E\in L(\X,\Z)$ with closed range and $A\colon\dom(A)\subseteq\X\to\Z$ be a closed and densely defined operator. If $(E,A)$ has a radiality index, then it also has a Weierstraß form, which is unique up to isomorphisms.
\end{thm}

\begin{proof}
    Let $p=p_{\rm rad}^{(E,A)}$ denote the radiality index. For arbitrary $\lambda_0,\ldots,\lambda_p$ define
    \begin{align*}
        \X^1 &= \overline{\ran (\lambda_0 E-A)\inv E \cdot\ldots\cdot (\lambda_p E-A)\inv E }^{\Vert \cdot\Vert_{\X}}, &
        \X^2 &= \ker (\lambda_0 E-A)\inv E \cdot\ldots\cdot (\lambda_p E-A)\inv E,\\
        \Z^1 &= \overline{\ran E(\lambda_0 E-A)\inv \cdot\ldots\cdot E (\lambda_p E-A)\inv}^{\Vert\cdot\Vert_{\Z}}, &
        \Z^2 &= \ker E (\lambda_0 E-A)\inv \cdot\ldots\cdot E(\lambda_p E-A)\inv.
    \end{align*}
    By \cite[Section 2.1, 2.2 \& 2.5]{sviridyuk_linear_2003} these spaces are independent of the choice of $\lambda_i$, $i\in\set{0,\ldots,p}$,
    \begin{align*}
        E_2 & \coloneqq E\vert_{\X^2} \in L(\X^2,\Z^2),\\
        A_2 & \coloneqq A\vert_{\dom(A_2)} \colon \dom(A_2)\subseteq \X^2\to\Z^2,\\
        \dom(A_2) & \coloneqq \dom(A)\cap\X^2
    \end{align*}
    is boundedly invertible operator with $A_2\inv \in L(\Z^2,\X^2)$, such that $A_2\inv E_2\in L(\X^2)$ and $E_2A_2\inv\in L(\Z^2)$ are nilpotent with degree less or equal $p+1$. Furthermore, there are two projections, namely
    \begin{align*}
        P\colon \X\to\X, \quad x\mapsto\lim_{\lambda\to\infty} (\lambda (\lambda E-A)\inv E)^{p+1} x
    \end{align*}
    onto $\X^1$ with $\ker P =\X^2$, $\ran P =\X^1$, $Px_1=x_1$ for all $x_1\in \X^1$ and
    \begin{align*}
        R\colon \Z\to\Z, \quad z\mapsto\lim_{\lambda\to\infty} (\lambda E (\lambda E-A)\inv )^{p+1} z
    \end{align*}
    onto $\Z^1$ with $\ker R = \Z^2$, $\ran R = \Z^1$, $Rz_1=z_1$ for all $z_1\in \Z^1$, such that $\X=\X^1\oplus\X^2$ and $\Z=\Z^1\oplus\Z^2$. Furthermore, $(\lambda E-A)\inv E x\in\dom(A)$ for all $x\in \dom(A)$ and thus, since $A$ is closed, $Px\in\dom(A)$. Hence,
    \begin{align}\label{APx=RAx}
        \begin{split}
            APx&= A \lim_{\lambda \to \infty} (\lambda (\lambda E-A)\inv E)^{p+1}x\\
            &=\lim_{\lambda \to \infty} A (\lambda (\lambda E-A)\inv E)^{p+1}x\\
            &= \lim_{\lambda \to \infty} (E\lambda (\lambda E-A)\inv)^{p+1} Ax\\
            &= RAx.
        \end{split}
    \end{align}
    Since $E$ is bounded we derive
    \begin{align}\label{EPx=REx}
        \begin{split}
            EPx&= E \lim_{\lambda \to \infty} (\lambda (\lambda E-A)\inv E)^{p+1}x\\
            &=\lim_{\lambda \to \infty} E (\lambda (\lambda E-A)\inv E)^{p+1}x\\
            &= \lim_{\lambda \to \infty} (E\lambda (\lambda E-A)\inv)^{p+1} Ex\\
            &= REx
        \end{split}
    \end{align}
    for all $x\in\X$. Now, we need to show that 
    \begin{align*}
        E_1 & \coloneqq E\vert_{\X^1}\in L(\X^1, \Z^1)\\
        \intertext{and}
        A_i & \coloneqq A\vert_{\dom(A_i)}\colon \dom(A_i)\subseteq \X^i\to\Z^i
    \end{align*}
    is closed and densely defined, whereby $\dom(A_i)\coloneqq \dom(A)\cap\Z^i$, $i\in\set{1,2}$. This was already proved for $p=0$ in \cite[Prop.~2.5]{jacob_solvability_2022} and under stronger radiality assumptions in \cite[Cor.~2.5.2, Thm.~2.5.3]{sviridyuk_linear_2003}. 
    Let $x\in \X^1$. Since $\X=\X^1\oplus\X^2$, $x\in \X^1$ if and only if $Px=x$. By \eqref{EPx=REx} we conclude
    \begin{equation*}
        Ex=EPx=REx \in \Z^1
    \end{equation*}
    and, since $E$ is bounded, the same yields for $E_1$.
    Analogously, using \eqref{APx=RAx} we have $Ax=APx=RAx\in \Z^1$ for all $x\in \dom(A_1)$. Now, let $x\in \X^1$. Thus, there exists a sequence $(x_n)_n\subseteq \dom(A)$ with $x_n\to x$ for $n\to\infty$. Since $Px_n\in\dom(A_1)$ and $Px_n\to Px=x$ for $n\to\infty$, one has $\overline{\dom(A_1)}=\X^1$. Similarly, using $I-P$ one can show $\overline{\dom(A_2)}=\X^2$. Since $A$ is closed and $A_i(\dom(A_i))\subseteq \X^i$, $A_i$ is closed as well, $i\in\{1,2\}$.
    
    What is left to show is, that $E_1$ is boundedly invertible. Since $E_1$ is bounded, $\ker E_1\subseteq \ker E\subseteq \X^2$ and $\X=\X^1\oplus\X^2$, it follows that $E_1$ is injective. What remains to show is the surjectivity of $E$, because then the assertion follows from the Closed Graph Theorem. Since the space $\Z^1$ does not depend on the choice of the $\lambda_i$, we can choose a $\lambda \in \rho(E,A)$, such that $\Z^1=\overline{\ran (E(\lambda E-A)\inv )^{p+2}}^{\Vert\cdot\Vert_{\Z}}$ \cite[Lem.~2.2.8]{sviridyuk_linear_2003}. Let $y\in \ran (E(\lambda E-A)\inv)^{p+2}$. Then there exists a $z\in\Z$ with 
    \begin{equation*}
        y = (E(\lambda E-A)\inv )^{p+2}z = E ((\lambda E-A)\inv E)^{p+1} (\lambda E-A)\inv z
    \end{equation*}
    and thus
    \begin{equation*}
        \ran (E(\lambda E-A)\inv)^{p+2} \subseteq E(\ran ((\lambda E-A)\inv E)^{p+1})\subseteq E(\X^1).
    \end{equation*}
    Since $\ran E$ is closed, the same yields for $E(\X^1) = \ran E_1 = \Z^1\cap \ran E$. Thus, the closure of $\ran(E(\lambda E-A)\inv)^{p+2}$ is still a subset of $E(\X^1)$. Since $\ran E_1\subseteq \Z^1 = \overline{\ran (E(\lambda E-A)\inv )^{p+2}}^{\Vert\cdot\Vert_{\Z}}$ one obtains the surjectivity of $E$. Define
    \begin{align*}
        \tilde P = \begin{bmatrix}P\\ I_{\X}-P\end{bmatrix} \in L(\X,\X^1\times \X^2), \quad \tilde R=\begin{bmatrix}R\\ I_{\Z}-R\end{bmatrix}\in L(\Z,\Z^1\times\Z^2).
    \end{align*}
    Then, 
    \begin{align*}
        \tilde P \inv = \begin{bmatrix}I_{\X^1} & I_{\X^2}\end{bmatrix} \in L(\X^1\times\X^2, \X), \quad \tilde R\inv = \begin{bmatrix}I_{\Z^1} & I_{\Z^2}\end{bmatrix} \in L(\Z^1\times\Z^2, \Z).
    \end{align*}
    Thus,
    \begin{subequations}
        \begin{align}
            (E,A) &\sim (\tilde R E \tilde P \inv, \tilde R A \tilde P \inv)\\
            &\sim \left( \begin{bmatrix} E_1 & 0 \\ 0 & E_2\end{bmatrix}, \begin{bmatrix} A_1 & 0 \\ 0 & A_2\end{bmatrix} \right)\\
            &\sim \left( \begin{bmatrix} I_{\Z^1} & 0 \\ 0 & E_2A_2\inv\end{bmatrix}, \begin{bmatrix} A_1E_1\inv & 0 \\ 0 & I_{\Z^2}\end{bmatrix} \right).\label{eqn:equiv-dae1}
        \end{align}
    \end{subequations}
    The uniqueness of the Weierstraß form follows from the uniqueness of the nilpotency index \cite[Prop.~6.2]{erbay_index_2024}, which exists since the nilpotency index of $(E,A)$ is at most $p_{\rm rad}^{(E,A)}+1$ \cite[Prop.~6.3]{erbay_index_2024}.
\end{proof}

\begin{remark}
    With \eqref{eqn:equiv-dae1} one can rewrite the DAE $\frac{d}{dt}Ex=Ax$ as follows
    \begin{equation}\label{eqn:wform-dEx=Ax}
        \frac{d}{dt}\begin{bmatrix} I_{\Z^1} & 0 \\ 0 & E_2A_2\inv\end{bmatrix} \begin{bmatrix}z_1(t) \\ z_2(t)\end{bmatrix} = \begin{bmatrix} A_1E_1\inv & 0 \\ 0 & I_{\Z^2}\end{bmatrix} \begin{bmatrix}z_1(t) \\ z_2(t)\end{bmatrix},
    \end{equation}
    on $\Z^1\times\Z^2$.
    Further, one can also show that
    \begin{equation*}
        (E,A)\sim \left( \begin{bmatrix} I_{\X^1} & 0 \\ 0 & A_2\inv E_2 \end{bmatrix}, \begin{bmatrix} E_1\inv A_1 & 0 \\ 0 & I_{\X^2}\end{bmatrix}\right)
    \end{equation*}
    and rewrite the DAE $E\frac{d}{dt}x = Ax$ in the following way
    \begin{align}
        \begin{bmatrix} I_{\X^1} & 0 \\ 0 & A_2\inv E_2\end{bmatrix} \frac{d}{dt} \begin{bmatrix}x_1(t) \\ x_2(t)\end{bmatrix} & = \begin{bmatrix} E_1\inv A_1 & 0 \\ 0 & I_{\X^2}\end{bmatrix} \begin{bmatrix}x_1(t) \\ x_2(t)\end{bmatrix},\label{eqn:wform-Edx=Ax}
    \end{align}
    on $\X^1\times\X^2$.
\end{remark}

In the following, we will recall and generalize an example from \cite[Ex.~6.5]{erbay_index_2024}, which shows that the radiality index can be greater than $0$.

\begin{ex}
    Let $\W$ be a Hilbert space and $(A_0, \dom(A_0))$ generate a $C_0$-semigroup on $\W$. For $b\in \dom(A_0)$ and $c \in \dom(A_0^\ast)$ with $\langle b, c\rangle\neq 0$ define the operators $B_u=bu$, where $u\in \C$ and $Cz=\langle c,z\rangle$ for any $z\in \W$. We define the following DAE on $\Z=\W\times \C$ by 
    \begin{equation}\label{eq:zero-dyn}
        \frac{d}{dt}\underbrace{\bem I & 0\\ 0 & 0\enm}_{\eqqcolon E} x(t) = \underbrace{\bem A_0 & B\\ C & 0\enm}_{\eqqcolon A} x(t), \quad t\geq 0. 
    \end{equation}
    In \cite{erbay_index_2024} it was already proven that the resolvent index is at most $2$ and for specific $A_0$, $B$ and $C$ it was also shown that the radiality index is $1$. We want to generalize the last statement a bit. To do this, we look again at the left- and right-$E$-resolvents.

    We want to show that the radiality index exists by using the Weierstraß form. Define $Q_b z = z-\frac{\langle c,z\rangle }{\langle c, b\rangle} b$. 
    Then, $Q_b b=0$, $\ran Q_b=\ker C$ and $\W=\W^1\oplus\W^2 \coloneqq \ker C \oplus \mathrm{span}(b)$. 
    For more concise notation, we additionally define $\tilde C = \frac{1}{\langle c,b\rangle}C$, $\tilde B = \frac{1}{\langle c,b\rangle} B$ and $K\coloneqq \tilde C A_0$. Note that $Kz\coloneqq \tilde C A_0 z = \frac{1}{\langle c,b\rangle} \langle c, A_0 z \rangle = \frac{1}{\langle c,b\rangle} \langle A_0^\ast c,  z \rangle$ holds for all $z\in W$ and therefore $K$ is a bounded operator on $\W$.
    Thus, we can define the isomorphisms $U\colon \W \times \C \to \W^1 \times \W^2\times \C$ and $V\colon \W^1\times\W^2\times\C\to\W\times \C$ via
    \begin{align*}
        & & U &\coloneqq \bem Q_b & -Q_b A_0 \tilde B \\ 0 & \tilde B \\ \tilde C & -K\tilde B\enm, & V &\coloneqq \bem I_{\W^1} & I_{\W^2} & 0\\ -K & 0 & 1\enm, & &
        \intertext{with inverses}
        & & U\inv &\coloneqq \bem I_{\W^1} & A_0 & B\\ 0 & C & 0\enm, & V\inv &\coloneqq \bem Q_b & 0\\ I_{\W}-Q_b & 0\\ KQ_b & 1\enm. & &
    \end{align*}
    These mappings applied to \eqref{eq:zero-dyn} we derive
    \begin{align*}
        \tilde E \coloneqq U \bem I_{\W} & 0 \\ 0 & 0\enm V= \bem I_{\W^1} & 0\\ 0 & N \enm, \qquad \tilde A \coloneqq U \bem A_0 & B\\ C & 0\enm V = \bem Q_b A_0 & 0 \\ 0 & I_{\W^2\times \C} \enm,
    \end{align*}
    with $N=\bem 0 & 0 \\ \tilde C & 0 \enm$. Since $A_0$ is a generator of a $C_0$-semigroup, $Q_b A_0 z = A_0 z - \frac{\langle c, A_0 z\rangle}{\langle c,b\rangle} b = A_0 z - BK z$, $z\in \W^1$ and $BK$ is bounded, $Q_bA_0$ also generates a $C_0$-semigroup. Moreover, $(\tilde E, \tilde A)$ has a Weierstraß form as in \eqref{eqn:dae-wform} and together with Theorem \ref{thm:2} b) $(\tilde E, \tilde A)$ is $1$-radial.
\end{ex}

Further examples of DAEs with existing radiality index (such as the linearized Navier-stokes equation) can be found in \cite{Fedorov2009OnSO}.
Next, the radiality index for DAEs in Weierstraß form gets investigated.

\begin{thm}\label{thm:2}
    Assume, that $(E,A)= \left(\Big[\begin{smallmatrix}
        I_{\Y^1} & 0\\ 0 & N
    \end{smallmatrix}\Big], \Big[\begin{smallmatrix}
        A_1 & 0\\ 0 & I_{\Y^2}
    \end{smallmatrix}\Big]\right)$, where $A_1$ and $N$ are defined as in \eqref{eqn:dae-wform}. Then
    \begin{enumerate}[\;\, a)]
        \item $(E,A)$ has radiality index $p_{\rm rad}^{(E,A)}=p\in\N$ if and only if $(I_{\Y^1}, A_1)$ has radiality index $p_{\rm rad}^{(E,A)}=p\in\N$ and $N$ has nilpotency degree smaller or equal to $p+1$.
        \item If $A_1$ generates a $C_0$-semigroup and $N$ has nilpotency degree smaller or equal to $p+1$, then $(E,A)$ is $p$-radial and, especially, $p_{\rm rad}^{(E,A)} = p$.
    \end{enumerate}
\end{thm}

\begin{proof} 
    Let $N$ has nilpotency degree $k+1\in \N$ and $\lambda \in \rho(E,A)$. Then, $\lambda \in \rho(I_{\Y^1}, A_1)$. Conversely, if $\lambda \in \rho(I_{\Y^1}, A_1)$, then $(\lambda N-I_{\Y^2})\inv = -\sum_{l=0}^k (\lambda N)^l$. Thus, 
    \begin{equation}\label{eqn:res-eq}
        \rho(E,A)=\rho(I_{\Y^1},A_1) = \rho(A_1).
    \end{equation}
    In particular, for $\lambda \in \rho(E,A)$ we obtain 
    \begin{equation}\label{eqn:lres=rres}
        (\lambda E- A)\inv  E = \begin{bmatrix}
        (\lambda I_{\Y^1}-A_1)\inv & 0 \\ 0 & (\lambda N-I_{\Y^2})\inv N
        \end{bmatrix} = E(\lambda E- A)\inv
    \end{equation}
    and
    \begin{equation*}
        \prod_{l=0}^k (\lambda_l E-A)\inv E=\Big[\begin{smallmatrix}
            \prod_{l=0}^k(\lambda_l I_{\Y^1}-A_1)\inv & 0\\ 0 & 0 
        \end{smallmatrix}\Big]
    \end{equation*}
    for $\lambda_0,\ldots,\lambda_p\in\rho( E, A)$. In particular, we have
    \begin{equation}\label{eqn:prod-res-norm}
        \left\Vert \prod_{l=0}^k(\lambda_l I_{\Y^1}-A_1)\inv\right\Vert = \left\Vert \prod_{l=0}^k (\lambda_l E-A)\inv E\right\Vert.
    \end{equation}
    Thus, Part a) follows directly from \eqref{eqn:res-eq} and \eqref{eqn:prod-res-norm}. To show Part b) assume, that $A_1$ generates a $C_0$-semigroup and $N$ has a nilpotency degree smaller or equal to $p+1$. Using the Theorem of Hille-Yosida there exists $\omega, C>0$ such that $(\omega, \infty)\subseteq \rho(I_{\Y^1}, A_1) = \rho(A_1)$ and $\Vert ((\lambda I_{\Y^1}-A_1)\inv)^n\Vert\leq \frac{C}{(\lambda-\omega)^n}$ for all $\lambda\geq \omega$, $n\in \N$. Together with \eqref{eqn:lres=rres} and \eqref{eqn:prod-res-norm} we have
        \begin{equation*}
            \left\Vert \left(\prod_{l=0}^p (\lambda_l E-A)\inv E\right)^n\right\Vert 
            = \left\Vert \left( \prod_{l=0}^p E(\lambda_l E- A)\inv \right)^n\right \Vert 
            \leq \frac{C^{p+1}}{\left(\prod_{l=0}^{p+1}\lambda_l-\omega\right)^n} 
        \end{equation*}
        for all $\lambda_0,\ldots,\lambda_{p+1}\geq\omega$. Thus, $(E, A)$ is $p$-radial. 
\end{proof}

A direct consequence of Theorem \ref{thm:2} is the equality of different index-terms introduced in \cite{erbay_index_2024}. Here, we will denote the differentiation index, the chain index and the perturbation index by $p_{\rm diff}^{(E,A)}$, $p_{\rm chain}^{(E,A)}$ and $p_{\rm pert}^{(E,A)}$. 

\begin{cor}
    Assume, that $(E,A)$ has a Weierstraß form and nilpotency index $p_{\rm nilp}^{(E,A)}$, i.e., $(E,A)\sim \left(\Big[\begin{smallmatrix}
        I_{\Y^1} & 0\\ 0 & N
    \end{smallmatrix}\Big], \Big[\begin{smallmatrix}
        A_1 & 0\\ 0 & I_{\Y^2}
    \end{smallmatrix}\Big]\right)$. If $A_1$ generates a $C_0$-semigroup on $\Y^1$. Then, 
    \begin{align*}
        p_{\rm rad}^{(E,A)}+1 
        = p_{\rm res\vphantom{d}}^{(E,A)} 
        = p_{\rm nilp}^{(E,A)} 
        = p_{\rm diff}^{(E,A)} 
        = p_{\rm chain}^{(E,A)} 
        = p_{\rm pert}^{(E,A)}.
    \end{align*}
\end{cor}

\begin{proof}
    This follows from Theorem \ref{thm:2} and \cite{erbay_index_2024}.
\end{proof}

Next, we want to investigate the connection between integrated semigroups and linear differential algebraic equations. 

\begin{defi}
    A linear operator $A$ on a Banach space $\X$ is the \textit{generator of an $(n-1)$-times integrated semigroup} $(S(t))_{t\geq 0}$, if there exists an $n\in\N$, $M,\omega>0$ and a strongly continuous family of operator $(S(t))_{t\geq 0}$ in $ L(\X)$ with $\left\Vert S(t)\right\Vert \leq M\e^{\omega t}$ for all $t\geq 0$, $(\omega,\infty)\in \rho(A)$ and
    \begin{align*}
        (\lambda I_\X -A)\inv x= \lambda^{n-1} \int_0^\infty \e^{-\lambda t}S(t)x\dx[t]
    \end{align*}
    for $x\in \X$. 
\end{defi}

It should be clear that for $n=1$ this coincides with the more known notion of a $C_0$-semigroup. The importance of integrated semigroups becomes clear when one is interested in the well-posedness of the Cauchy problem
\begin{equation}\label{ACP}
    \begin{cases}
        \frac{d}{dt}x(t)=Ax(t),\quad t\geq 0,\\
        x(0)=x_0,
    \end{cases}
\end{equation}
for an $x_0\in\X$. Because, then for all $x_0\in \dom(A^n)$ the unique solution of \eqref{ACP} is given by
\begin{equation*}
    x(t)=S(t)A^{n-1}x_0 + \sum_{k=0}^{n-2} \frac{1}{k!} t^k A^k x_0
\end{equation*}
\cite[eq.~(4.1)]{neubrander_integrated_1988}. Furthermore, for all $x_0 \in \X$ the function $t\mapsto S(t)x$ is a solution of the $n$-times integrated Cauchy problem and, thus, a \textit{mild} solution of \eqref{ACP}.

With these information we can illustrate the relation between the complex resolvent of a DAE and the generator of an integrated semigroup, given that $(E,A)$ has a Weierstraß form.

\begin{thm}\label{thm:3} 
    Assume, that $(E,A)$ has a Weierstraß form as in \eqref{eqn:dae-wform}. If $A_1$ is densely defined and $(E,A)$ has a complex resolvent index $p_{\rm c, res}^{(E,A)}$, then $A_1$ generates an (at least) $p_{\rm c, res}^{(E,A)}+2$-times integrated semigroup. 
\end{thm}

\begin{proof}
    Assume, that $(E,A)\sim \left(\Big[\begin{smallmatrix}
        I_{\Y^1} & 0\\ 0 & N
    \end{smallmatrix}\Big], \Big[\begin{smallmatrix}
        A_1 & 0\\ 0 & I_{\Y^2}
    \end{smallmatrix}\Big]\right)$ has a complex resolvent index $p=p_{\rm c,res}^{(E,A)}$. Then, by \eqref{eqn:res-eq} $\rho(E,A)=\rho(A_1)$ and
    \begin{equation}
        \Vert (\lambda I_{\Y^1}-A_1)\inv \Vert 
        \leq \tilde C \Vert (\lambda E-A)\inv \Vert 
        \leq C \Vert\lambda\Vert^{p-1}, \quad\lambda\geq\omega,
    \end{equation}
    for given $C, \tilde C>0$ and $\omega >0$. Thus, $(I_{\Y^1},A_1)$ has a complex resolvent index and by \cite[Cor.~4.9]{neubrander_integrated_1988} $A_1$ generates an (at least) $p+2$-times integrated semigroup.
\end{proof}

\begin{cor}
    Assume, that $(E,A)$ has a complex radiality index. Then $A_1E_1\inv$ and $E_1\inv A_1$ from \eqref{eqn:wform-dEx=Ax} and \eqref{eqn:wform-Edx=Ax} generate integrated semigroups.
\end{cor}
\begin{proof}
    In \cite[Prop.~5.4]{erbay_index_2024} it was shown that the existence of the radiality index already implies the existence of the resolvent index. This implication extends naturally to the complex resolvent and complex radiality index. Therefore, if $(E,A)$ has a complex radiality index, then by Theorem \ref{thm:1} $(E,A)\sim \left(\Big[\begin{smallmatrix}
        I_{\Z^1} & 0\\ 0 & E_2A_2\inv
    \end{smallmatrix}\Big], \Big[ \begin{smallmatrix}
        A_1E_1\inv & 0\\ 0 & I_{\Z^2}
    \end{smallmatrix}\Big]\right) \sim \left(\Big[\begin{smallmatrix}
        I_{\X^1} & 0\\ 0 & A_2\inv E_2
    \end{smallmatrix}\Big], \Big[ \begin{smallmatrix}
        E_1\inv A_1 & 0\\ 0 & I_{\X^2}
    \end{smallmatrix}\Big]\right)$ and by Theorem \ref{thm:3} $A_1E_1\inv$ and $E_1\inv A_1$ generate integrated semigroups.
\end{proof}

\section{Infinite dimensional port-Hamiltonian DAEs}\label{section-4}
Let $\X, \Z$ be Hilbert spaces, $E, Q\in L(\X,\Z)$ and $A\colon\dom(A)\subseteq \Z\to\Z$ be a closed and densely defined operator. Then, the following differential-algebraic equation 
\begin{equation}\label{eqn:dae-phs}
    \begin{cases}
        \frac{d}{d t} Ex(t) &=AQx(t), \quad t\geq 0,\\
        Ex(0) &=z_0,
    \end{cases}
\end{equation}
for a $z_0\in \Z$ is called a \textit{port-Hamiltonian differential-algebraic equation} (short pH-DAE), if $A$ is dissipative, $\ran E$ is closed, $Q$ is invertible and 
\begin{equation}\label{eqn:dae-eq=qe}
    E^\ast Q = Q^\ast E\geq 0
\end{equation}
holds. In this context, we call
\begin{equation}\label{hamiltonian}
    H(x)\coloneqq \langle Ex,Qx\rangle_\Z,\quad x\in \X
\end{equation}
the \textit{Hamiltonian} of $(E,AQ)$.
Note, that here $T\colon \X\to\X$ is called \textit{non-negative}, denoted by $T\geq 0$, if $\langle Tx,x\rangle_\X\geq 0$ for all $x\in \X$. 

These systems provide a framework for modeling and analyzing energy-dissipating physical systems, including electrical circuits, mechanical systems, fluid dynamic, thermal systems or robotics and control systems. In this Section we want to examine a few properties of pH-DAEs and show that the derivative of the Hamiltonian dissipates on all classical solutions of \eqref{eqn:dae-phs}. 

\begin{remark}\label{remark:1}
    \begin{enumerate}[a)]
        \item One can define a pH-DAE similarly for the case
        \begin{equation}\label{eqn:dae-phs2}
            \begin{cases}
                E\frac{d}{dt}x(t) &=AQx(t), \quad t\geq 0,\\
                x(0) &=x_0,
            \end{cases}
        \end{equation}
        for $x_0\in \X$. Hence, $x$ a \textit{classical solution} of \eqref{eqn:dae-phs2}, if $x(t)$ is continuously differentiable on $\R_{\geq 0}$, $x(t)\in\dom(AQ)$ for all $t\geq 0$ and \eqref{eqn:dae-phs2} holds.
        \item It can be observed that \eqref{eqn:dae-eq=qe} implies
        \begin{equation}\label{eqn:dae-eqinv=qinve}
            EQ\inv = Q\ainv E^\ast \geq 0.
        \end{equation}
        In fact, using \eqref{eqn:dae-eq=qe} we derive $EQ\inv = Q\ainv Q^\ast EQ\inv = Q\ainv E^\ast$, with $Q\ainv = (Q\inv)^\ast = (Q^\ast)\inv$. Applying the positivity of $EQ^\ast$ we have
        \begin{equation*}
            \langle Q\ainv E^\ast z, z\rangle_\Z = \langle Q\ainv E^\ast Q Q\inv z, z\rangle_\Z = \langle E^\ast Q (Q\inv z), (Q\inv z)\rangle_\X \geq 0, \quad z\in\Z.
        \end{equation*}
        \item \label{remark:Q=IandX=Z} Since $Q$ is invertible it is possible to assume $\Z=\X$ and $Q=I_\Z$, such that $E$ becomes non-negative and self-adjoint. This results from defining $z(t)\coloneqq Qx(t)$ in \eqref{eqn:dae-phs}, such that $\frac{d}{dt}Ex(t)=AQx(t)$ is equivalent to $\frac{d}{dt} EQ\inv z(t)=Az(t)$ and $EQ\inv$ being non-negative and self-adjoint as seen in \eqref{eqn:dae-eqinv=qinve}. Similarly, $E\frac{d}{dt}x(t)=AQx(t)$ is equivalent to $E\frac{d}{dt}Q\inv z(t)=Az(t)$.
        \item By c) \eqref{eqn:dae-phs} is equivalent to $\frac{d}{dt}EQ\inv z(t) = Az(t)$. Since $\ran E$ is closed, the same holds for $\ran EQ\inv$ and one can split up the space into $\Z=\ran EQ\inv \oplus \ker EQ\inv \eqqcolon \Z^1\oplus\Z^2$. Thus, $\frac{d}{dt}EQ\inv z(t) = Az(t)$ is equivalent to
        \begin{equation*}
            \frac{d}{dt}\bem E_1 & 0\\0 & 0\enm \bem z_1(t)\\ z_2(t)\enm = \bem A_1 \\ A_2\enm \bem z_1(t)\\ z_2(t)\enm, \quad t\geq 0,
        \end{equation*}
        whereby $A_i\colon \Z \to \Z^1$ and $E_1\coloneqq EQ\inv_{\ran EQ\inv}$. Thus, the representation for port-Hamiltonian differential algebraic equations chosen here corresponds with the notion of an abstract Hamiltonian differential-algebraic equation defined in \cite[Assumption 7]{mehrmann_abstract_2023}.
        Furthermore, it should be mentioned that there are more ways to define and approach port-Hamiltonian systems, such as using relations \cite{gernandt_linear_2021} or Dirac-structures \cite{van_der_schaft_generalized_2018}. 
    \end{enumerate}
\end{remark}

Next, we want to show that the derivative of the Hamiltonian dissipates on solutions, i.e. for a solution of \eqref{eqn:dae-phs}, $\frac{d}{dt}H(x(t))\leq 0$ for all $t\geq 0$. 

\begin{prop}\label{BTB=B}
    Let $B\in L(\Z)$ be a non-negative and self-adjoint operator on $\Z$ and assume that $\ran B$ is closed. Then there exists a $c>0$ and $T\in L(\Z)$ with $T\geq cI_\Z$ and 
    \begin{equation*}
        BTB=B.
    \end{equation*}
    Furthermore, $\langle \cdot, \cdot \rangle_T\coloneqq \langle T \cdot, \cdot \rangle$ induces a norm on $\Z$ equivalent to $\Vert \cdot\Vert_\Z$.
\end{prop}

Obviously, if $B$ is invertible then $T=B\inv$.

\begin{proof}
    Since $\ran B$ is closed and $B$ is self-adjoint, $\Z$ has an orthogonal decomposition into $\Z=\ran B\,\oplus\,\ker B$. Let $P_{\ran B}$ and $P_{\ker B}$ be the projections of $\Z$ onto $\ran B$ and $\ker B$. Define $\tilde B \colon \ran B \to \ran B$, $\tilde B(z) = Bz$. This is a bounded, positive selfadjoint operator on the Hilbert space $\ran B$, which is bijective (note that $\ker \tilde B = \ker B \cap \ran B = \{0\}$). 
    Hence, since $\tilde B^{-1}$ is bounded, $\tilde B$ is strictly positive with 
    \begin{equation*}
        \langle \tilde B z_1 , z_1 \rangle\geq c \|z_1\|^2, \quad z_1\in \ran B
    \end{equation*}
    for all $z_1 \in \ran B$, for some $c>0$.

    Let $\iota_{\ran B}\colon \ran B\to\Z$ be the embedding and define 
    \begin{equation*}
        T\colon \iota_{\ran B}\tilde B\inv \iota_{\ran B}^\ast + P_{\ker B}.
    \end{equation*}
    This operator satisfies $BTB=B$ and is bounded. Let $z= z_1+ z_2 \in \Z=\ran B \, \oplus \ker B$. Then
    \begin{align*}
        \langle  Tz, z \rangle &= \langle \tilde B  z_1 + z_2  , z_1 + z_2 \rangle \\
        &= \langle \tilde B z_1 , z_1 \rangle+ \langle z_2 ,z_2 \rangle\\
        &\geq c \|z_1\|^2 + \|z_2\|^2 .
    \end{align*}
    Thus, $T$ is strictly positive. The self-adjointness of $T$ follows from that of $B$, completing the proof.
\end{proof}

\begin{cor}\label{cor:1}
    Let $E,Q\in L(\X,\Z)$, whereby $E$ has a closed range, $Q$ is invertible and \eqref{eqn:dae-eq=qe} holds. Then, there exists a $c>0$ and $T\in L(\Z)$ with $T\geq cI_\Z$ and 
        \begin{equation}\label{ETE=EQ=QE}
            E^\ast TE = E^\ast Q = Q^\ast E.
        \end{equation}
        Furthermore, $\langle \cdot, \cdot\rangle_T \coloneqq\langle T\cdot, \cdot\rangle$ induces a norm on $\Z$ equivalent to $\Vert \cdot\Vert_\Z$.
\end{cor}

\begin{proof}
    This follows from Remark \ref{remark:1}, Proposition \ref{BTB=B} with $B=EQ\inv$ and the fact that $Q$ is invertible.
\end{proof}

\begin{thm}\label{thm:6}
    Let $(E,AQ)$ be a pH-DAE with Hamiltonian $H$. Then, for all classical solutions $x\colon \R_{\geq 0}\to\X$ 
    \begin{equation*}
        \frac{d}{dt}\langle E x(t), Qx(t)\rangle\leq 0, \quad t\geq 0.
    \end{equation*}
\end{thm}

\begin{proof}
    By Corollary \eqref{cor:1} we know that $T\geq cI_Z$, such that \eqref{ETE=EQ=QE} holds. Let $x\colon \R_{\geq 0}\to\X$ be a classical solution of \eqref{eqn:dae-phs}. Then 
    \begin{align*}
        \frac{d}{dt}\langle Ex(t), Qx(t) \rangle_\Z &= \frac{d}{dt} \langle x(t), \underbrace{E^\ast Q}_{=E^\ast T E} x(t)\rangle_\Z \\
        &= 2\Real \langle Ex(t), \frac{d}{dt} TEx(t)\rangle_\Z \\
        &= 2\Real \langle Qx(t), \frac{d}{dt}Ex(t)\rangle_\Z\\
        &= 2\Real \langle Qx(t), AQx(t)\rangle_\Z \leq 0, \quad t\geq 0,
    \end{align*}
    where we used the continuity of $E$, $Q$ and $T$ in the second and third equation. Note, that we used $\langle Ex(t) = TEx(t)\rangle = \langle x(t), E^\ast T E x(t) \rangle = \langle x(t), Q^\ast Ex(t) \rangle = \langle Qx(t), Ex(t)\rangle$ in the third equality.
\end{proof}

\begin{remark}\label{remark:2}
    From the proof of Theorem \ref{thm:6} it becomes clear, that for all classical solutions $x$ of \eqref{eqn:dae-phs} the DAE is already dissipating, i.e.~one has
    \begin{equation*}
        \Real \langle E x(t), AQx(t)\rangle_T \leq 0, \quad t\geq 0.
    \end{equation*}
    It is possible to show the same for the adjoint system and, if the set containing all trajectories of the solutions of \eqref{eqn:dae-phs} is closed, one can apply \cite[Thm.~3.7]{jacob_solvability_2022} to gain a Weierstraß form on a subset. To be more precise, we know that $\ran E^\ast$ is closed and $Q\inv$ is invertible. Thus, there also exists a $S\in L(\X)$ with $S\geq \tilde c I_\X$, for a $\tilde c>0$, with
    \begin{equation*}
        ESE^\ast = EQ\inv = Q\ainv E^\ast
    \end{equation*}
    and $\langle\cdot,\cdot\rangle_{S\inv} = \langle S\inv \cdot, \cdot \rangle$ induces a norm on $\X$ equivalent to $\Vert \cdot \Vert_\X$. Define $\X_{S\inv}\coloneqq (\X, \langle\cdot,\cdot\rangle_{S\inv})$ and $\Z_T\coloneqq (\Z, \langle\cdot,\cdot\rangle_{T})$. Let $z\in \set{z\in\Z \,\middle|\, Tz\in\dom(A^\ast)}$ and $x\in \dom(AQ)$. Then, 
    \begin{equation*}
        \langle AQx,z\rangle_T = \langle x, SQ^\ast A^\ast Tz\rangle_{S\inv}.
    \end{equation*}
    In fact, by simple reformulations it is easy to see, that
    \begin{equation*}
        \dom((AQ)^{\ast}_{S\inv, T})=\set{z\in\Z \,\middle|\, \exists x^\ast \in \X, \forall x \in \dom(AQ)\colon \langle z,AQx\rangle_T=\langle x^\ast,x\rangle_{S\inv} },
    \end{equation*}
    where $(AQ)^{\ast}_{S\inv, T}$ denotes the adjoint of $AQ\colon \dom(AQ)\subseteq \X_{S\inv}\to\Z_T$.
    Consider now the \textit{adjoint-system} of \eqref{eqn:dae-phs}
    \begin{equation}\label{adjoint-system:1}
        \begin{cases}
            \frac{d}{dt} E^\ast z(t) &= Q^\ast A^\ast z(t), \quad t\geq 0,\\
            E^\ast z(0) &= x_0\in \X.
        \end{cases}
    \end{equation}
    Then, by choosing $z(t)=T\tilde z(t)$, this is equivalent to
    \begin{equation}\label{adjoint-system:2}
        \begin{cases}
            \frac{d}{dt} (E)^{\ast}_{S\inv, T} \tilde z(t) &= (AQ)^{\ast}_{S\inv, T} \tilde z(t), \quad t\geq 0,\\
            (E)^{\ast}_{S\inv, T} \tilde z(0) &= \tilde x_0\in \X.
        \end{cases}
    \end{equation}
    Let $z$ be a classical solution of \eqref{adjoint-system:2}. Assuming that $A^\ast$ is dissipative, one has
    \begin{align*}
        \Real\langle (E)^{\ast}_{S\inv, T} z(t), (AQ)^{\ast}_{S\inv, T} z(T)\rangle_{S \inv} &= \Real\langle E^\ast T z(t), SQ^\ast A^\ast Tz(T)\rangle_\X \\
        &= \Real\langle E^\ast T z(t), \frac{d}{dt} S E^\ast T z(T)\rangle_\X\\
        &= \Real\langle T z(t), \frac{d}{dt} E S E^\ast T z(T)\rangle_\X\\
        &= \Real\langle Q\inv T z(t), S\inv \frac{d}{dt} S E^\ast  T z(T)\rangle_\X\\
        &= \Real\langle Q\inv T z(t), S\inv S Q^\ast A^\ast  T z(T)\rangle_\X\\
        &= \Real\langle T z(t), A^\ast T z(T)\rangle_\X \leq 0, \quad t\geq 0.
    \end{align*} 
    Thus, if $U\coloneqq \{u\in\Z\,\vert\, \exists x$ classical solution of \eqref{eqn:dae-phs} $\exists t\geq 0\colon x(t)=u \}\subseteq \Z$ and $V \coloneqq \{v\in\Z\,\vert\, \exists z$ classical solution of \eqref{adjoint-system:2} $\exists t\geq 0\colon z(t)=v \}\subseteq \Z$ are both closed sets, one can apply \cite[Thm.~3.7]{jacob_solvability_2022} to $(E\vert_U, (AQ)\vert_U)$ on the spaces $\X_{S\inv}$ and $\Z_T$. 
\end{remark}

\begin{ex}
    Consider longitudinal vibrations in a viscoelastic nanorod. Let $l$ be the length of the nanorod, $N(x,t)$ be the resultant force of axial stress, $w(x,t)$ be the displacement of the nanorod in $x$ direction, $C$ be the elastic modulus, $D$ be the cross sectional area, $\mu$ be a nonlocal parameter, $\rho$ the mass density, $\tau_d$ the viscous damping and $a^2$, $b^2$ be the stiffness and damping coefficients of the light viscoelastic layer and consider the system introduced in \cite{karlicic_nonlocal_2015}. Consider
    \begin{subequations}
        \begin{align}
            \frac{\partial N(x,t)}{\partial x} &= a^2 w(x,t) + b^2 \frac{\partial w(x,t)}{\partial t} + \rho D \frac{\partial^2 w(x,t)}{\partial t^2},\\
            N(x,t)-\mu\frac{\partial^2 N(x,t)}{\partial x^2} &= CD\left(\frac{\partial w(x,t)}{\partial x} + \tau_d \frac{\partial^2 w(x,t)}{\partial x\partial t}\right),
        \end{align}
    \end{subequations}
    with boundary conditions
    \begin{equation}
        \frac{\partial w}{\partial t}(0,t) = \frac{\partial w}{\partial t}(l,t)=0.
    \end{equation}
    In \cite{heidari_port-hamiltonian_2019} the associated port-Hamiltonian system is given through
    \begin{equation}\label{eq:heidari-zwart}
        \underbrace{\begin{bsmallmatrix}
            1 & 0 & 0 & 0 & 0\\
            0 & 1 & 0 & 0 & 0\\
            0 & 0 & 1 & 0 & 0\\
            0 & 0 & 0 & 1 & 0\\
            0 & 0 & 0 & 0 & 0
        \end{bsmallmatrix}}_{=E}\frac{dz(x,t)}{dt} = \underbrace{\begin{bsmallmatrix}
            0 & 1 & 0 & 0 & 0\\
            -1 & -b^2 & 0 & 0 & \frac{\partial}{\partial x}\\
            0 & 0 & -CD\tau_d -\mu b^2 & -1 & 1\\
            0 & 0 & 1 & 0 & 0\\
            0 & \frac{\partial }{\partial x} & -1 & 0 & 0
        \end{bsmallmatrix}}_{=A} \underbrace{\begin{bsmallmatrix}
            a^2 & 0 & 0 & 0 & 0\\
            0 & \frac{1}{\rho D} & 0 & 0 & 0\\
            0 & 0 & \frac{1}{\mu \rho D} & 0 & 0\\
            0 & 0 & 0 & CD + \mu a^2 & 0\\
            0 & 0 & 0 & 0 & 1
        \end{bsmallmatrix}}_{=Q} z(x,t)
    \end{equation}
    and state 
    \begin{equation}
        z(x,t)=\begin{bmatrix}
            w(x,t)\\
            \rho D \frac{\partial w(x,t)}{\partial t}\\
            \mu \rho D \frac{\partial^2w(x,t)}{\partial x\partial t}\\
            \frac{\partial w(x,t)}{\partial x}\\
            N(x,t)
        \end{bmatrix}.
    \end{equation}
    Here, $E$ and $Q$ are bounded operator living on $\X=\Z=L^2((0,l); \R^5)$ and $A\colon\dom(A)\subseteq \X\to\X$ with 
    \begin{equation*}
        \dom(A)\coloneqq\set{ (z_1,\ldots,z_5)^T \in \X \mid z_2, z_5 \in H^1(0,l), z_2(0)=z_2(l)=0 }.
    \end{equation*}
    In \cite{heidari_port-hamiltonian_2019} the existence of solution was studied by reducing the system to a homogeneous port-Hamiltonian system (see \cite[Ch.~7]{jacob_linear_2012}). Here, we provide another approach by simply examining the dissipativity of the DAE as seen in Remark \ref{remark:2}.
    
   Since the various physical constants in $Q$ are positive, it is easy to show that $E^\ast Q$ is non-negative and self-adjoint and with
    \begin{align*}
        \langle Az,z\rangle_\X &= \int_0^l z_2(x)\frac{\partial z_5(x)}{\partial x} + z_5(x)\frac{\partial z_2(x)}{\partial x} - b^2 z_2^2(x) - (CD\tau_d + \mu b^2)z^2_3(x)\dx\\
        &= -b^2 \left\Vert z_2\right\Vert_\X^2 - (CD\tau_d + \mu b^2)\left\Vert z_3\right\Vert_\X^2 \leq 0, \quad z=(z_1,\ldots,z_5)\in \dom(A),
    \end{align*}
    $A$ is dissipative. Thus, $(E,AQ)$ fits into our definition of a port-Hamiltonian DAE. 
    
    Now, as in \cite[p.~452]{heidari_port-hamiltonian_2019} we impose the boundary conditions directly on the space and show, that \eqref{eq:heidari-zwart} has radiality index $0$. Define 
    \begin{equation*}
        \X_t\coloneqq \set{ (z_1,\ldots,z_5)^T \in L^2((0,l);\R^5) \mid z_2, z_5 \in H^1(0,l), z_2(0)=z_2(l)=0, \mu\frac{\partial z_2(x)}{\partial x} = z_3(x) }.
    \end{equation*} with the inner product $\langle \cdot , \cdot \rangle_{\tilde \X_t}\coloneqq \langle Q\cdot,\cdot\rangle_{\X_t}$. Since $Q$ is coercive this induces a norm equivalent to $\Vert\cdot\Vert_{\X}$. $\X_t$ is in fact closed \cite[Lem.~4.1]{heidari_port-hamiltonian_2019}. Define $\tilde \X \coloneqq (\X, \langle \cdot , \cdot \rangle_{\tilde \X_t})$. Then, 
    \begin{align}\label{eq:1}
        \Real \langle AQz, Ez\rangle_{\tilde \X_t} 
        &\leq \Real \int_0^l \frac{1}{\rho D} \left(z_2(x)\frac{\partial z_5(x)}{\partial x} + z_5(x)\frac{\partial z_2(x)}{\partial x} \right)\dx=0
    \end{align}
    holds for all $z=(z_1,\ldots,z_5)\in \dom(AQ)$. Furthermore, $(AQ)^\ast = A^\ast Q$ in $\X_t$ and, through similar calculations, one derives $\Real\langle A^\ast Q z, Ez\rangle_{\X_t}\leq 0$ for all $z\in\dom((AQ)^\ast)$. Then, by \cite[Thm.~3.6, 3.7]{jacob_solvability_2022} $(E,AQ)$ is $0$-radial and admits a decomposition as in \eqref{eqn:wform-dEx=Ax} on $\X_t = \X_t^1\oplus\X_t^2$ with $E_2(AQ)_2\inv = 0$ and $(AQ)_2\inv E_2=0$. In particular, $(AQ)_1E_1\inv$ generates a contraction semigroup on $\X_t^1$. 
\end{ex}

Note that in this example the operator $T$ from above is already given by $Q$, since $E^\ast Q E = E^\ast Q = Q^\ast E$.

By \cite[Thm.~3.7]{jacob_solvability_2022} it becomes clear that if $(E,AQ)$ is a pH-DAE with $\Real \langle AQx, Ex\rangle \leq 0$ and $\Real \langle (AQ)^\ast z, E^\ast z\rangle\leq 0$, then $(E,AQ)$ is in fact $0$-radial (and therefore has radiality index $0$), admits a Weierstraß form as in \eqref{eqn:wform-dEx=Ax} or \eqref{eqn:wform-Edx=Ax} and $(AQ)_1E_1\inv$ and $E_1\inv (AQ)_1$ generate contraction semigroups on $\tilde \Z^1$ and $\tilde \X^1 $. In this case, $(AQ)_1E_1\inv$ and $E_1\inv (AQ)_1$ also generate $C_0$-semigroups on $\Z^1$ and $\X^1$.

\section{Weierstraß form and solutions of port-Hamiltonian DAEs}\label{section-5}
It is possible to show that the (complex) resolvent index for pH-DAEs always exists and that it is bounded by $2$ ($3$). This is an already widely known result in finite dimensions and was shown for the infinite dimensional case in \cite{erbay_index_2024} for $\X=\Z$ and $Q=I$. 

\begin{prop}\label{res-leq-2}
    Let $E$, $Q$ and $A$ be defined as before, such that $(E,AQ)$ defines a pH-DAE. Assume that there exists a $\omega>0$ with $(\omega,\infty)\subseteq \rho(E,AQ)$ $(\C_{\Real >\omega}\subseteq \rho(E,AQ))$. Then the (complex) resolvent index exists and is at most $2$ (at most $3$).
\end{prop}

\begin{proof}
    Using Remark \ref{remark:1} and \cite[Thm.~3.3]{erbay_index_2024} one knows that $(EQ\inv, A)$ is a pH-DAE with an existing (complex) resolvent index, which is at most $2$ (at most $3$), and together with $(\lambda EQ\inv -A)\inv=Q(\lambda E-AQ)\inv$ for all $\lambda \in \rho(EQ\inv, A) = \rho(E,AQ)$ the assertion follows. 
\end{proof}

\begin{prop}
    Let $E$, $Q$ and $A$ be defined as before, such that $(E,AQ)$ defines a pH-DAE and assume, that there exists a $\omega>0$ with $\C_{\Real >\omega}\subseteq \rho(E,AQ)$. Then, for every $x_0 \in \ran((\mu E-AQ)\inv E)^{p_{\mathrm{c,res}}^{(E,AQ)}+2}$ there exists a unique solution $x\colon \R_{\geq 0}\to\X$ of the following differential algebraic equation
    \begin{align*}
        \begin{cases}
            \frac{d}{dt}Ex(t) &= AQx(t), \quad t\geq 0,\\
            x(0) &=x_0.
        \end{cases}
    \end{align*}
\end{prop}
\begin{proof}
    This is a direct consequence of Theorem \ref{thm:7} and Proposition \ref{res-leq-2}.
\end{proof}

The next goal is to examine the well-posedness of a pH-DAE on the whole domain, given that the radiality index exists. Let $(E,AQ)$ be such a pH-DAE. By Theorem \ref{thm:1} \eqref{eqn:dae-phs} is equivalent to
\begin{equation}\label{eqn:dae-phs-wform}
    \frac{d}{dt}\begin{bmatrix}
        I_{\Z^1}\inv & 0\\
        0 & E_2(AQ)_2\inv
    \end{bmatrix} \begin{bmatrix}
        z_1(t) \\ z_2(t)
    \end{bmatrix} = \begin{bmatrix}
        (AQ)_1E_1\inv & 0 \\
        0 & I_{\Z^2}
    \end{bmatrix} \begin{bmatrix}
        z_1(t) \\ z_2(t)
    \end{bmatrix}, \quad t\geq 0
\end{equation}
and \eqref{eqn:dae-phs2} is equivalent to
\begin{equation}\label{eqn:dae-phs2-wform}
    \begin{bmatrix}
        I_{\X^1}\inv & 0\\
        0 & (AQ)_2\inv E_2
    \end{bmatrix} \frac{d}{dt} \begin{bmatrix}
        x_1(t) \\ x_2(t)
    \end{bmatrix} = \begin{bmatrix}
        E_1\inv(AQ)_1 & 0 \\
        0 & I_{\X^2}
    \end{bmatrix} \begin{bmatrix}
        x_1(t) \\ x_2(t)
    \end{bmatrix}, \quad t\geq 0.
\end{equation}
The main problem with a transformation like that is, that \eqref{eqn:dae-phs-wform} and \eqref{eqn:dae-phs2-wform} are not necessarily pH-DAEs anymore. Thus, the resulting question is under which circumstances $(AQ)_1E_1\inv$ and $E_1\inv(AQ)_1$ generate a $C_0$-semigroup. A possible condition for that is $AQ$ being strongly $(E,p)$-radial \cite[Thm.~2.6.1]{sviridyuk_linear_2003}.

\begin{thm}\label{thm:4}
    Let $E$, $Q$ and $A$ be defined as before, such that $(E,AQ)$ defines a pH-DAE. Assume, that the radiality index exists and that $\C_{\Real \geq \omega}\subseteq \rho(E,AQ)$ for a $\omega >0$. Then, $(AQ)_1E_1\inv$ and $E_1\inv(AQ)_1$ generate an (at least) $p_{\rm c,res}^{(E,AQ)}+2$-times integrated semigroup. If additionally
    \begin{itemize}
        \item[a)] 
            $Q^\ast (\Z^1) = \X^1$, then $E_1\inv (AQ)_1$ generates a $C_0$-semigroup on $\X^1$.
        \item[b)] 
            $Q(\X^1)=\Z^1$, then $(AQ)_1 E_1\inv$ generates a $C_0$-semigroup on $\Z^1$. 
    \end{itemize}
\end{thm}

\begin{proof}
    Let $x\in \X$ and $\lambda\in\rho(E,AQ)$ with $\lambda>0$. Thus, $(\lambda E-AQ)\inv \in L(\Z, \X)$. As seen in the proof of Theorem \ref{thm:1} $(AQ)_1E_1\inv$ and $E_1\inv (AQ)_1$ are densely defined and by Proposition \ref{res-leq-2} the complex resolvent index $p_{\rm c,res}^{(E,AQ)}$ exists. Thus, by Theorem \ref{thm:3} $(I_{\Z^1}, (AQ)_1E_1\inv)$ and $(I_{\X^1}, E_1\inv (AQ)_1)$ generate an (at least) $p_{\rm c, res}^{(E,AQ)}+2$-times integrated semigroup.

    In order to show a), assume that $Q^\ast (\Z^1) = Q^\ast (E(\X^1))= \X^1$ holds. Thus, $Q^\ast E\colon \X^1\to\X^1$ is non-negative and self-adjoint. In fact, since $Q^\ast E$ is self-adjoint the mappings $Q^\ast E \pm \i I_\X$ are surjective. Hence, for every $z\in\X^1$ there are $x_\pm = x_\pm^1+x_\pm^2 \in\X=\X^1\oplus\X^2$ with $(Q^\ast E\pm \i I_\X)x_\pm = z$. Together with $Q^\ast E(\X^1)=\X^1$ one has $PQ^\ast E x_\pm = PQ^\ast E x_\pm^1 + PQ^\ast Ex_\pm^2 = Q^\ast E Px_\pm^1 = Q^\ast E Px_\pm$ and
    \begin{align*}
        z &= Pz\\
        &= PQ^\ast E x_\pm \pm \i Px_\pm \\
        &= Q^\ast E Px_\pm \pm \i Px_\pm = (Q^\ast E \pm \i I_{\X^1}) x_\pm. 
    \end{align*}
    Thus, $Q^\ast E \pm \i I_{\X^1}\colon \X^1\to\X^1$ are surjective and consequently, as a closed operator, $(Q^\ast E)\vert_{\X^1}$ is self-adjoint. Furthermore, since $Q^\ast$ and $E_1$ are invertible and $(Q^\ast E)\vert_{\X^1}$ is self-adjoint, $(Q^\ast E)\vert_{\X^1}$ becomes invertible and $((Q^\ast E)\vert_{\X^1})\inv$ self-adjoint as well. Hence, using Lemma \ref{BTB=B} for $B=((Q^\ast E)\vert_{\X^1})\inv$ there exists a $c>0$, such that $T=((Q^\ast E)\vert_{\X^1})\geq cI_{\X^1}$ and $\langle \cdot, \cdot\rangle_{\tilde \X^1}\coloneqq \langle (Q^\ast E)\vert_{\X^1} \cdot ,\cdot\rangle$ induces a norm equivalent to $\Vert \cdot\Vert_\X$ on $\X^1$. Define $\tilde \X^1\coloneqq (\X^1, \langle \cdot, \cdot\rangle_{\tilde\X^1})$. Since $A$ is dissipative one derives
    \begin{equation*}
        \Real \langle E_1\inv(AQ)_1 x, x\rangle_{\tilde\X^1} 
        = \Real \langle AQx, Qx\rangle_\X
        \leq 0, \quad x\in \dom(E_1\inv (AQ)_1),
    \end{equation*}
    which means that $E_1\inv(AQ)_1$ is dissipative in $\tilde \X^1$. Given that $\X=\X^1\oplus\X^2$, $E_1$ is bijective and $(\lambda E_1 -(AQ)_1)$ is surjective, $(\lambda I_{\X^1}-E_1\inv(AQ)_1) = E_1\inv(\lambda E_1-(AQ)_1)\colon\dom((AQ)_1)\subseteq \X^1\to\X^1$ becomes surjective as well.
    By the Theorem of Lumer-Phillips \cite[Thm.3.4.5]{arendt_vector-valued_2011} $E_1\inv(AQ)_1$ generates a contraction semigroup on $\tilde\X^1$. In this case, $E_1\inv(AQ)_1$ also generates a $C_0$-semigroup on $\X^1$.

    To show b) assume, that $Q(\X^1)=Q(E\inv (\Z^1))=\Z^1$ holds. As seen before it is possible to show, that $QE\inv$ is non-negative, self-adjoint and has a non-negative and self-adjoint inverse, such that $QE\inv\geq cI_{\X}$ for a $c>0$. Define $\langle\cdot,\cdot\rangle_{\tilde\Z^1}\coloneqq \langle QE\inv \cdot, \cdot\rangle$ and $\tilde\Z^1 \coloneqq (\Z^1, \langle \cdot,\cdot\rangle_{\tilde\Z^1})$. Thus,
    \begin{align*}
        \Real \langle (AQ)_1E_1\inv z, z \rangle_{\tilde\Z^1} = \Real\langle A QE_1 z, QE_1 z\rangle \leq 0,\quad z\in\dom((AQ)_1E_1\inv)
    \end{align*}
    since $A$ is dissipative. Then, the rest of this proof follows the previous part.
\end{proof}

\begin{remark}
    Let $\X=\Z$ and $Q=I_\X$ and assume that $E$ commutes with $A$ on $\dom(A)$. Thus
    \begin{equation*}
        Ex=E(\lambda E-A)(\lambda E-A)\inv x = (\lambda E-A)E (\lambda E-A)\inv x
    \end{equation*}
    or equivalently $(\lambda E-A)\inv Ex=E(\lambda E-A)\inv x$. Thus, the left- and right-$E$ resolvent coincide and $\X^1=\Z^1$, $\X^2=\Z^2$ as well as the projections from the projections from Theorem \ref{thm:1} $P=R$. In this case $Q^\ast E(\X^1)=\X^1$ and $QE\inv(\Z^1)=\Z^1$ hold. Unfortunately, the assumption that $E$ and $AQ$ commute massively limits the choice of systems. Because, if one considers the following type of system
    \begin{align*}
        \frac{d}{dt} \underbrace{\begin{bmatrix}
            I & 0\\
            0 & N
        \end{bmatrix}}_{\eqqcolon E} \begin{bmatrix}
            x_1(t) \\ x_2(t)
        \end{bmatrix} = \underbrace{\begin{bmatrix}
            A_1 & A_2 \\
            A_3 & A_4
        \end{bmatrix}}_{\eqqcolon A} \begin{bmatrix}
            x_1(t) \\ x_2(t)
        \end{bmatrix}, \quad t\geq 0,
    \end{align*}
    and additionally assume that $EA=AE$ hold, then $A_2=A_2N$, $NA_3=A_3$ and $NA_4=A_4N$.
\end{remark}

\begin{ex} 
    Consider the system
    \begin{equation*}
        \frac{d}{dt}\underbrace{\begin{bmatrix}
            I & 0\\
            0 & 0
        \end{bmatrix}}_{=\mathcal{E}} \begin{bmatrix}
            x_1(t)\\ x_2(t)
        \end{bmatrix} = \underbrace{\begin{bmatrix}
            A_1 & A_2\\
            A_3 & A_4
        \end{bmatrix}}_{=\mathcal{A}} \begin{bmatrix}
            x_1(t)\\ x_2(t)
        \end{bmatrix}
    \end{equation*}
    with existing radiality index and $\mathcal{A}$ being dissipative.\\ 
    Since $\ker E\subseteq \X^2$ and $\X=\X^1\oplus\X^2$, one has $\mathcal{E}_1=I_{\X^1}$. Thus, one has $\mathcal{E}(\X^1)= \X^1 = \Z^1 = \mathcal{E}_1(\Z^1)$. But, in this case it is even possible to say directly something about the generation of $C_0$-semigroups without using Theorem \ref{thm:4}. Because in such a case $\mathcal{A}_1 = \mathcal{A}_1\mathcal{E}_1\inv = \mathcal{E}_1\inv\mathcal{A}_1$ is dissipative (as a restriction of a dissipative operator) and $\lambda I_{\X^1}-\mathcal{A}_1$ is surjective (since $(\omega,\infty)\subseteq\rho(\mathcal{E},\mathcal{A})=\rho(I_{\X^1}, \mathcal{A}_1)$). Hence, using the Lumer-Phillips Theorem one obtains the same outcome. This result is similar to \cite[Section IV]{jacob_solvability_2022}.
\end{ex}

\begin{ex} 
   Recall the system from \cite[Ex.~3.4]{erbay_index_2024}. Define $A={\rm diag\,} (A_0, A_1, A_2,\ldots)$ with 
    \begin{align*}
        A_0=\bem 0 & -1\\ 1& 0\enm,\quad A_k=\bem 0 & \sqrt{k^4+1} \\ -\sqrt{k^4+1} & -2\enm, \quad k\in \N,
    \end{align*}
    and $\dom(A)\coloneqq \set{x\in \ell^2 \mid Ax\in\ell^2}$. Then $A$ can be extended to $\ell^2$, which will denoted by $A_{-1}$. Define $E \in L(\ell^2)$, $B\colon\R\to \dom(A^\ast)'$ and $C\colon \dom(A)\to\R$ with $E={\rm diag\,}(E_0, E_1, E_2,\ldots)$ and $B=(B_0, B_1, B_2, \ldots)^T = C^\ast$, where
    \begin{align*}
        & & E_0 &= \bem 1 & 0\\ 0 & 0\enm, & E_k &= \bem 1 & 0 \\ 0 & 1\enm, \quad k\in\N, & & \\
        & & B_0 &= \bem 0\\ 1\enm, & B_k &=\bem 0\\ k^{\frac{5}{4}}\enm, \quad k\in\N. & & 
    \end{align*}
    Define then the system
    \begin{align*}
        \frac{{\rm d}}{{\rm d}t}\underbrace{\bem E & 0 & 0\\ 0 & 0 & 0\\ 0 & 0 & 0\enm}_{\mathcal{E}} \bem x_1\\ x_2\\ x_3\enm = \underbrace{\bem A_{-1} & B & 0\\ -C & 0 & I\\ 0 & -I & 0\enm}_{\mathcal{A}} \bem x_1\\ x_2 \\ x_3 \enm
    \end{align*}
    on $\X=\ell^2\times \R\times \R$. Obviously, $\mathcal{E}$ is non-negative and self-adjoint and, by its construction, $\mathcal{A}$ with maximal domain is dissipative. Furthermore, for $\lambda\in\rho(\mathcal{E}, \mathcal{A})$ 
    \begin{equation*}
        (\lambda\mathcal{E}-\mathcal{A})\inv = \bem (\lambda E-A)\inv & 0 & (\lambda E-A)\inv B \\ 0 & 0 & I \\ -C(\lambda E-A)\inv & -I & C(\lambda E-A)\inv B\enm.
    \end{equation*}
    It was shown in \cite[Ex.~3.4]{erbay_index_2024} that this system has real resolvent index $2$ and complex resolvent index $3$.
    
   We now compute the radiality index of $(\mathcal{E}, \mathcal{A})$. First, calculate
    \begin{align*}
        (\lambda\mathcal{E}-\mathcal{A})\inv \mathcal{E} &= \bem (\lambda E-A)\inv E & 0 & 0 \\ 0 & 0 & 0 \\ -C(\lambda E-A)\inv E & 0 & 0\enm, \\
        \mathcal{E} (\lambda\mathcal{E}-\mathcal{A})\inv &= \bem E(\lambda E-A)\inv & 0 & E(\lambda E-A)\inv B \\ 0 & 0 & 0 \\ 0 & 0 & 0 \enm.
    \end{align*}
    It is needed to determine the growth rate of $(\lambda E-A)\inv = {\rm diag}(M_0(\lambda ), M_1(\lambda ),$ $ M_2(\lambda ),\ldots)$ with
    \begin{equation*}
         M_0(\lambda ) = \bem 0 & -1\\ 1& \lambda \enm, \quad M_k(\lambda ) = \frac{1}{\lambda (\lambda +2)+k^4+1}\bem \lambda +2 & \sqrt{k^4+1} \\ -\sqrt{k^4+1} & \lambda \enm, \quad k\in\N.
    \end{equation*}
    It is easy to see that $M_0(\lambda )$ has a linear growth and $M_k(\lambda )$ is decreasing with rate $\frac{1}{\lambda }$. Hence, due to the structure of $E$, $(\lambda E-A)\inv E$ and $E(\lambda E-A)\inv$ are bounded and
    \begin{equation*}
        \left\Vert (\lambda E-A)\inv E (\mu E-A)\inv E \right\Vert = \left\Vert E(\lambda E-A)\inv E (\mu E-A)\inv \right\Vert \leq \frac{C}{\lambda \mu}, \quad \lambda, \mu>0 
    \end{equation*}
    for a given $C>0$. Moreover, by simple calculations its is easy to show that $B_k^T M_k(\lambda)$ and $M_k(\lambda)B_k$ are still decreasing with rate $\frac{1}{\lambda}$. Thus, for all $\lambda, \mu >0$
    \begin{align*}
        \left\Vert C(\lambda E-A)\inv E (\mu E-A)\inv E \right\Vert &\leq \frac{C}{\lambda \mu}, \\
        \left\Vert E(\lambda E-A)\inv E (\mu E-A)\inv B \right\Vert &\leq \frac{C}{\lambda \mu}.
    \end{align*}
    Since 
    \begin{align*}
         (\lambda\mathcal{E}-\mathcal{A})\inv \mathcal{E} (\mu\mathcal{E}-\mathcal{A})\inv \mathcal{E} &= \bem (\lambda E-A)\inv E(\mu E-A)\inv E & 0 & 0 \\ 0 & 0 & 0 \\ -C(\lambda E-A)\inv E(\mu E-A)\inv E & 0 & 0\enm, \\
         \mathcal{E} (\lambda\mathcal{E}-\mathcal{A})\inv \mathcal{E} (\mu\mathcal{E}-\mathcal{A})\inv &= \bem E(\lambda E-A)\inv E(\mu E-A)\inv & 0 & E(\lambda E-A)\inv E(\mu E-A)\inv B \\ 0 & 0 & 0 \\ 0 & 0 & 0 \enm
    \end{align*}
    the system $(\mathcal{E}, \mathcal{A})$ has radiality index $1$.
\end{ex} 



\section*{Acknowledgements}
The financial support of NSERC (Canada) for the research described in this paper is gratefully acknowledged. 
The authors would like to thank Hannes Gernandt and Timo Reis for valuable discussions.\\

\noindent
\textbf{Data Availability} No data was used for the research described in the article.

\section*{Declarations}

\noindent
\textbf{Conflicts of interest} On behalf of all authors, the corresponding author states that there is no conflict of interest.


\end{document}